\documentclass[12pt]{article}
\usepackage{graphicx} 
\usepackage{amsmath,amssymb,enumerate,xcolor, epsfig}
\usepackage{amsfonts}
\usepackage[maxnames=4]{biblatex}
\usepackage{url}
\usepackage{soul}

\title{If a  Minkowski billiard is projective, it is the standard billiard}
\author{Alexey Glutsyuk and Vladimir S. Matveev\thanks{ https://orcid.org/0000-0002-2237-1422 }} 
\date{}

\newtheorem{theorem}{Theorem}
\newtheorem{remark}{Remark}

\newtheorem{proposition}{Proposition}
\newtheorem{example}{Example}
\newenvironment{proof}{{\noindent \textbf{Proof}\,\,}}{\hspace*{\fill}$\Box$\medskip}
\def\rr{\mathbb R}
\def\rp{\mathbb{RP}}

\def\mcn{\mathcal N}
\def\gg{\mathfrak{g}}
\def\var{\varepsilon}
\def\wt{\widetilde}

\newcommand{\weg}[1]{}
\newcommand{\dedication}[1]{%
  \begin{flushright}
    \itshape
    #1
  \end{flushright}
}

\usepackage{biblatex} 
\begin{document}

\maketitle
\dedication{
To academician Anatolii  Fomenko \\  on the occasion of his 80th birthday}

\tableofcontents 

\begin{abstract} In the recent paper 
\cite{Glu24},
the first author of this note proved that if  a  billiard  in a convex domain in $\mathbb{R}^n$
    is simultaneously projective and  Minkowski, then  it is the  standard Euclidean billiard in an appropriate Euclidean structure. The proof was quite complicated and required high smoothness.  Here we present a direct simple proof of this result which works in $C^1$-smoothness. In addition we prove  the semi-local and local versions of the result  


\end{abstract}

\section{Introduction}

Projective
and  Finsler  billiards were introduced  respectively by S. Tabachnikov   in \cite{Tab97} and  by   E.  Gutkin and S.  Tabachnikov in \cite{GT02}. The Finsler metric corresponding to  Minkowski  billiards considered in the present note is the so-called Minkowski Finsler metric, that is, the corresponding Finsler function is a Minkowski norm. We call  such billiards  {\it Minkowski billiards}.    We recall the definitions  of projective billiards  and Minkowski  billiards below. These billiards  
generalise the {\it standard billiards} in convex domains, when the angle of reflection is equal to the angle
of incidence.

Additional motivation to study   Minkowski billiards  comes from symplectic geometry. Indeed, they are closely related to Viterbo's Conjecture \cite{viterbo}, which  was recently disproved in \cite{ostrover} by using   Minkowski  billiards,  and Mahler's Conjecture  from convex geometry \cite{ako1}. For more details see e.g. the paper \cite{Glu24} and its bibliography. 

  In \cite{Glu24}, it was asked when a Minkowski  billiard is a projective one. It was proved there that this happens if and only if  the corresponding Minkowski norm is an Euclidean norm, in which case the billiard is the standard billiard. The proof is quite complicated,  requires
  nontrivial asymptotic calculations and $C^6$-smoothness. In this note we give,
   see Theorems  \ref{thm:main}    and \ref{thloc}, an easy direct proof which works in the  $C^1$-smoothness. We also prove a generalised  ``local'' version, see Theorem \ref{thgen}.

  Let us now   recall  the   definitions. A {\it projective billiard} is determined by a convex compact body  $K\subset \mathbb{R}^n $ with $n\ge 2$, whose boundary is denoted by $\partial K$ and is assumed to be 
  $C^1$-smooth, and a field $H$ of nontrivial linear involutions 
  $H_q:T_q\rr^n\to T_q\rr^n$ with $q\in \partial K$ such that  $H_q 
  $ has eigenvalues $1$ and $-1$ and such that the tangent hyperplane to $\partial K$ at $q$ is the eigenspace corresponding to the eigenvalue $1$.  

  The reflection law is defined as follows: the incoming and outgoing velocity vectors  $v_{\textrm{in}}$ and $v_{\textrm{out}}$ 
  of the billiard  trajectory at the reflection point $q\in \partial K$ are 
  related by  $v_{\textrm{out}}= H_q(v_{\textrm{in}}).$

  The above definition of projective billiard is equivalent to the definition given in \cite{Tab97}, where a projective billiard was defined by a transversal line field $\mcn$ on $\partial K$: for every point $q\in\partial K$ the corresponding line $\mcn_q$ of the latter field is the eigenspace 
  of the operator $H_q$ corresponding to the eigenvalue $-1$. 
  \begin{figure}[ht]
  \begin{center}
   \epsfig{file=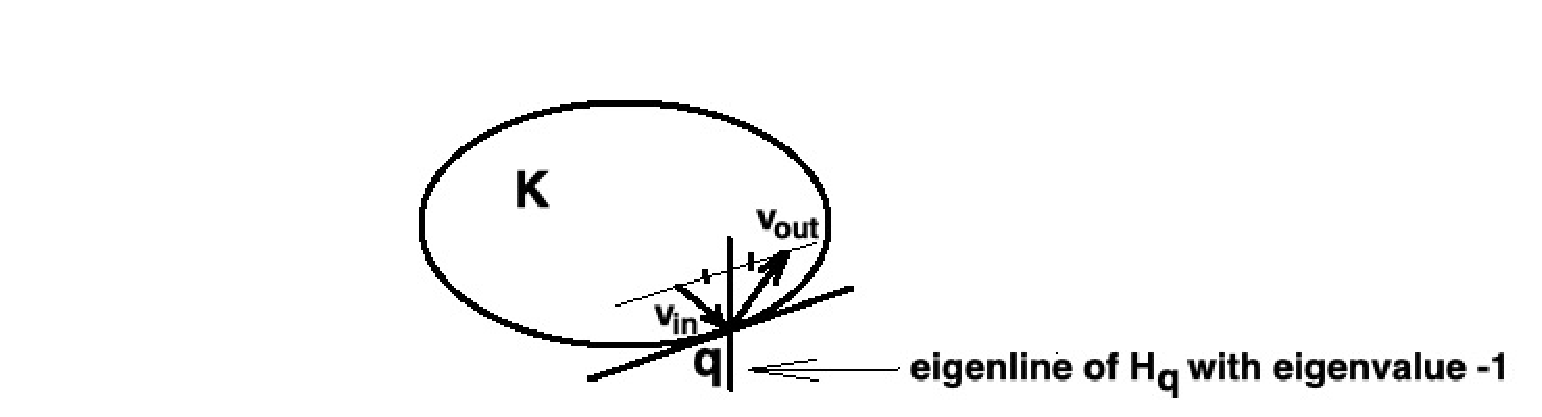, width=35em}
    \caption{A projective billiard.}
    \label{fig:1}
  \end{center}
\end{figure}

\begin{example}
      The standard billiard is a projective billiard such that at every $q\in \partial K$ the eigenvector of  $H_q$ corresponding to the  eigenvalue $-1$ is  normal   to $\partial K$.
\end{example}

 In what follows we identify the tangent spaces $T_q\rr^n$ with $\rr^n$ by translations. This identifies the  involution family $H_q$ 
  with a family of linear involutions $\rr^n\to\rr^n$, i.e., $n\times n$-matrices,  also denoted by $H_q$. Conversely, a family of non-trivial linear involutions $H_q$ satisfying assumptions above  uniquely determines a projective billiard structure on $\partial K$.

  A {\it Minkowski  billiard}\footnote{also called {\it $T$-billiard}} is determined by two strictly convex compact bodies $K\subset \mathbb{R}^n$ and $T\subset \mathbb{R}^{n*}$ with  $C^1$-smooth boundary, where $\mathbb{R}^{n*}$ is the dual vector space to  $\mathbb{R}^{n}$. We think without loss of generality that the origins 
  of $\mathbb{R}^n$  and $\mathbb{R}^{n*}$ coincide with the barycentres of $K$ and $T$, respectively, and denote by $F$ and $\widetilde F$ the Minkowski  norms corresponding to the domains $K$ and $T$, respectively. That is, $F$ and $\wt F$ are positively 1-homogeneous, i.e., $F(\lambda v)=\lambda F(v)$, $\wt F(\lambda v)=\lambda\wt F(v)$ for every $\lambda\in\mathbb R_+$ and every vector $v$, and  
  $$K=\{q\in \mathbb{R}^n \mid F(q)\leq 1\}  \ \ \textrm{ and }
T=\{p\in \mathbb{R}^{n*} \mid \wt F(p)\leq 1\}.   $$
Note that, in the initial definition of  Finsler and   Minkowski  billiards  given in    \cite{GT02}, it is assumed  in addition that $T$ is centrally-symmetric, that is, $\wt F(-v)=\wt F(v)$ for 
every $v\in\mathbb{R}^{n*}$. The corresponding Minkowski metrics are called {\it reversible} in Finlser geometry. Our definition of Minkowski billiards is due to  \cite{AO2} and allows irreversible Minkowski metrics. 
  
   The billiard motion lives in $K$ and 
     the reflection law at the point $q\in \partial K$ 
     is defined as follows.  Consider the differential $d_qF \in \mathbb{R}^{n*}$.  
     View the incoming velocity vector $v_{\textrm{in}}$ as an element of $(\mathbb{R}^{n*})^*$ via the canonical identification of $\mathbb{R}^n $  and $(\mathbb{R}^{n*})^*$,
 and take the point $p_1 \in \partial T\subset   \mathbb{R}^{n*}$  such that at this point $d_{p_1}\wt F$ is proportional
 to $v_{\textrm{in}}$ with a positive coefficient.   This condition determines $p_1$ uniquely. Next, take the line $\Lambda_q(p_1)$  parallel to $d_qF$ and  containing $p_1$:
 $$\Lambda_q(p_1):=\{p_1 + t\cdot d_qF \mid  t\in \mathbb{R}\} \subset \mathbb{R}^{n*}.$$
 If the line $\Lambda_q(p_1)$ intersects $\partial T$ in two points, 
  denote by $p_2$ its second point of intersection with $\partial T$. Then, as the vector $v_{\textrm{out}}$ we take the vector corresponding to $d_{p_2}\wt F$ viewed as a vector of $\mathbb{R}^n $ 
 via the canonical identification of $\mathbb{R}^{n} $ and $(\mathbb{R}^{n*})^*$. In the case, when the line $\Lambda_q(p_1)$ is tangent to $\partial T$, so $p_1$ is the unique point of their intersection, we set $v_{ \textrm{out}}= v_{\textrm{in}}$. 
 \begin{figure}[ht]
  \begin{center}
   \epsfig{file=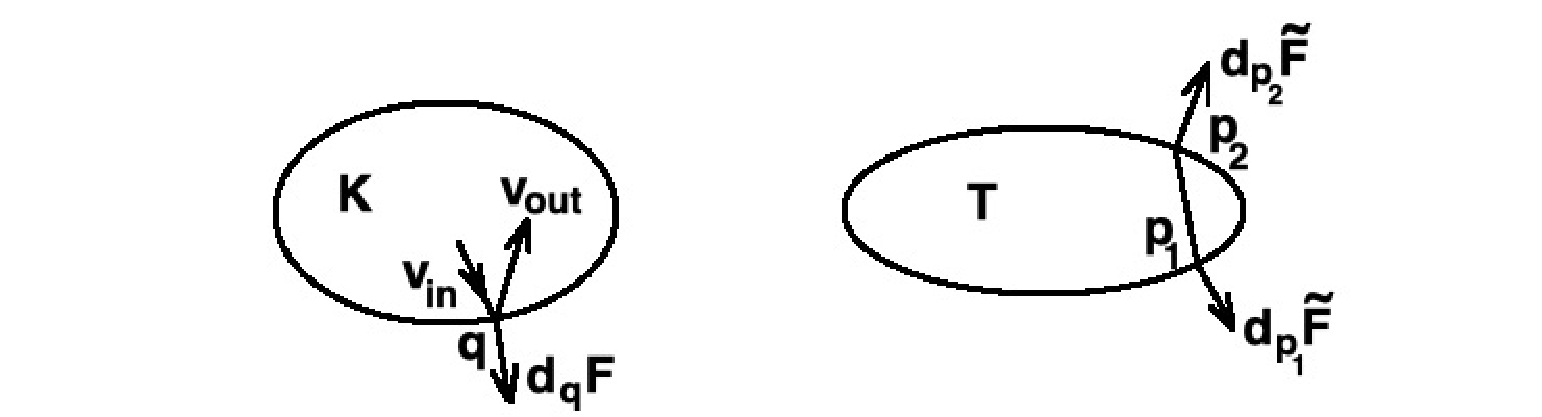, width=35em}
    \caption{A Minkowski billiard.}
    \label{fig:2}
  \end{center}
\end{figure}

\begin{example} \label{ex:2}
    If $T$ is an ellipsoid, then the Minkowski billiard is the standard billiard in the Euclidean structure such that the polar dual  of the 1-ball is a translation of this ellipsoid. In more detail, let $B:\mathbb R^{n*}\to \mathbb R^{n*}$ 
  be a linear transformation that sends $T$ to a ball. Let $B^*:\mathbb R^n\to\mathbb R^n$ be its dual, which is defined by the condition that for every $\alpha\in \mathbb R^{n*}$ and $v\in\mathbb R^n$ one has $(B(\alpha))(v)=\alpha(B^*(v))$. Then   $(B^*)^{-1}$ sends orbits of the  Minkowski billiard in $K$ to orbits of the standard billiard in its  image $(B^*)^{-1}(K)$. \end{example}

The  regularity assumptions in Theorem \ref{thm:main} below is that $\partial K$ and $\partial T$ have $C^1$-smooth boundary and strictly convex\footnote{That is, every line 
intersects  the boundary in at most two points}. These assumptions are clearly necessary for the reflection law  to be well-defined. Note that if a convex body has a differentiable boundary\footnote{That is, locally, there exists an injective  differentiable  mapping  from a small $(n-1)$-dimensional neighborhood of the origin to $\mathbb{R}^n$ or $\mathbb{R}^{n*}$ whose  Jacobi matrix has  maximal rank $n-1$ and  whose image coincides with the  intersection of a small neighborhood of the image of the origin with the 
boundary of the body}, then the boundary is $C^1$-smooth. Recall that strict convexity is equivalent to the property that every supporting hyperplane has precisely one common point
with  the body. 
The field  $H_q$ has no regularity assumption, it does not even need to be continuous in $q$.  

Our main ``global''
result is the following Theorem:
 \begin{theorem}\label{thm:main}
     Suppose for a convex body $T\subset \mathbb{R}^{n*}$ there exists a convex body $K$ equipped with a projective billiard structure on $\partial K$ such that the trajectories of the Minkowski billiard defined by $T$ coincide (up to reparametrization) with trajectories of the projective billiard. 
     Then, $T$ is an ellipsoid. 
 \end{theorem}

 We also give a simple proof of the following  ``semi-local'' version of Theorem \ref{thm:main}. To state it, let us consider germs of strictly convex $C^1$-smooth hypersurfaces $S=\partial T\subset\rr^{n*}$ at a point $O_S$ 
 and  $U=\partial K\in\rr^n$ at a point $O_U$. 
 We consider that $U$ and $S$ are level hypersurfaces of local $C^1$-smooth functions $F$ and $\wt F$ respectively without critical points, and the line through $O_S$ parallel to the vector $d_{O_U} F$ is tangent to $S$ at $O_S$. Then for every $q\in U$ close to $O_U$ the above Finsler billiard map $v_{\textrm{in}}\mapsto v_{\textrm{out}}$ is well-defined on all the incoming vectors $v_{in}\in T_qU$ close enough to the translation image to $T_qU$ of the vector $d_{O_S}\wt F$. 
 
\begin{theorem} \label{thloc} Let in the above conditions the Finsler billiard map $v_{\textrm{in}}\mapsto v_{\textrm{out}}$ be given 
by  a projective billiard structure on $U$: there exists a family of 
linear involutions $H_q:T_q\rr^n\to T_q\rr^n$, $q\in U$, with eigenvalues $\pm1$, that fix all the vectors tangent to $U$ at $q$, 
such that for every $q\in U$ and for every $v_{in}\in T_qU$ close enough to the translation image of the vector $d_{O_S}\wt F$ the vector  $v_{\textrm{out}}$  is proportional to $H_q(v_{\textrm{in}})$ with a positive coefficient.  Then,  $S$ is a quadric.
\end{theorem}

Theorem \ref{thloc} with hypersurface $S$ of class $C^6$ and with 
positive definite second fundamental form  is presented in an equivalent form as \cite[theorem 1.12]{Glu24}. Its original proof given in \cite{Glu24} is quite complicated and required a nontrivial result stating that if for a given hypersurface $S$ ``sufficiently many"  planar sections  are conics, then $S$ is a quadric: see \cite[theorem 3.7]{Glu24}, \cite[theorem 4]{art}  and a stronger result in \cite[section 3]{berger}. The proof in the current note  is  more  simple, uses only elementary arguments,  and is  essentially  self-contained. 

Of course, Theorem \ref{thm:main} is a corollary of Theorem \ref{thloc}. We will still first prove, in Section \ref{sec:3},  Theorem 1, as its proof introduces two  important constructions  which will be used also in our proof of Theorem \ref{thloc}. See also  Remark \ref{rem:2}, where we commented on it and reformulated the initial  problem as a problem about the body $T$ (respectively, the hypersurface $S$)
only; the reformulation  does not require  $K$ and $U$ anymore. 

The proof of Theorem \ref{thloc}  requires additional ideas, and needs to consider separately the cases $n=2$ and $n\ge 3$, which we do in Sections \ref{sec:4} and  \ref{sec:5},  respectively. 

\begin{example} If $S$ is a quadric, the above Finsler billiard map $v_{\textrm{in}}\mapsto v_{\textrm{out}}$ is indeed given by a projective billiard structure on 
$U$. See \cite[corollary 1.9]{Glu24}. 

Indeed, the space of strictly convex quadrics $S$ (ellipsoids, strictly convex hyperboloids and paraboloids) is a connected open subset, and the corresponding Finsler reflections depend analytically on $S$. Projectivity of Finsler reflection holds if $S$ are ellipsoids, see Example \ref{ex:2}, which form an open subset of quadrics, 
and it remains valid under analytic extension.
\end{example}

We also prove the following  generalization  of Theorems \ref{thm:main} and \ref{thloc}, which is in a sense their  local version. It deals with two germs of hypersurfaces $S_1,S_2\subset\rr^{n*}$ at two distinct points $O_1$ and $O_2$ respectively (playing the role of the boundary $\partial T$), a germ of hypersurface $U\subset\rr^n$ at a point $O_U$, a vector $v_{in}\in T_{O_U}\rr^n$. We consider that 
$U$, $S_1$, $S_2$ are level hypersurfaces of germs of $C^1$-smooth functions $F$, $\wt F_1$, $\wt F_2$, respectively. 
Consider the one-dimensional vector subspace $\alpha_0$ parallel to  $O_1O_2$ in $\rr^{n*}$  and  assume that it is parallel to   $ d_{O_U}F \in  \rr^{n*} $ and $O_1O_2$ is transversal to $S_1$ and $S_2$. We assume that the vector $v_{in}$ is proportional to $d_{O_1}\wt F_1$.
 
For every one-dimensional vector subspace $\alpha\subset\rr^{n*}$ 
close enough to $\alpha_0$,  consider the  involution 
$I_\alpha$ permuting  the points of intersection of 
each line parallel to $\alpha$ with $S_1$ and $S_2$. 
Next,  impose the 
following

\vspace{1ex}
\noindent {\bf Projectivity assumption} {\it For every $\alpha$ close enough to $\alpha_0$ there exists a linear involution $H_\alpha:\rr^n\to\rr^n$ such that for every $p\in S_1$ 
the image $H_{\alpha}(d_{p}\wt F_1)$ is proportional to 
$d_{I_\alpha(p)}\wt F_2$.} 

\begin{theorem} \label{thgen} Let $S_1,S_2\subset \rr^{n*}$ be two distinct germs of hypersurfaces at points $O_1$ and $O_2$ respectively. Let $\alpha_0\subset\rr^{n*}$ be a one-dimensional vector subspace. Let the line through $O_1$  parallel to  $\alpha_0$ 
pass through $O_2$ and be transversal to $S_1$ and $S_2$.
Let $\alpha_0$ 
satisfy the above projectivity assumption.  Then $S_1$ and $S_2$ are parts of the same quadric (allowed to be a union of distinct hyperplanes).
\end{theorem}

Of course, Theorem  \ref{thloc}, and therefore, Theorem   \ref{thm:main},  follow from Theorem \ref{thgen}. In our paper, we still first prove Theorems \ref{thm:main}, \ref{thloc} and then  Theorem \ref{thgen}. The reason is that  the first part of the proof of Theorem  \ref{thm:main} contains one of two main ideas necessary for the proof of Theorem \ref{thgen}, and the proof of  Theorem \ref{thloc} contains the second main idea. It is easier to explain these main ideas in the set up of Theorems \ref{thm:main} and \ref{thloc}, as one can use more geometric arguments.

  \subsection*{Acknowledgement } V.\,M. thanks the DFG (projects 455806247 and 529233771), and the ARC  (Discovery Programme DP210100951) for their support. A.\,G. thanks Friedrich Schiller  University Jena for partial support of his visit to Jena.  Both authors participated in the  
  program ``Mathematical Billiards: At the Crossroads of Dynamics, Geometry, Analysis, and Mathematical Physics'' of  Simons Center for Geometry and Physics at Stony Brook University. We wish to thank the center   and the organisers of the program for hospitality and partial  support. 
  Our paper answers  the question which aroused in  discussions with  A.\,Sharipova and S.\,Tabachnikov during this program. We   thank them and Yu.\,S.\,Ilyashenko for helpful discussions.  V.\,M. is grateful to T.\,Wannerer for explaining related elementary but necessary statements in convex geometry.

\section{Proof of Theorem \ref{thm:main}} \label{sec:3}
{ Assume a projective billiard  in $K$ with the field 
 $H$ of $n\times n$ matrices at the points of $\partial K$ has   the same trajectories as the  Minkowski  billiard in $K$ constructed by $T\subset \mathbb{R}^{n*}$.  We will work in notations of Introduction and in particular identify 1-forms on $\mathbb{R}^n$ with vectors of 
$\mathbb{R}^{n*}$ and one-forms on $\mathbb{R}^{n*}$ with vectors of  $\mathbb{R}^n$. We  assume   that  the origins in $\mathbb{R}^n$ and $\mathbb{R}^{n*}$  are the barycenters of   $K$ and $T$, respectively. 

In Section \ref{sec:2.1} we construct a family of  affine involutions $A_q, \, q\in \partial K$  that 
preserve the convex body $T$. From the construction it will be clear, that the corresponding linear transformations  are precisely the dual transformations to $H_q$ which we denote by $H^*_q$.  Moreover,  for different $q_1,q_2\in \partial K$ the corresponding $A_{q_1}, A_{q_2}$ are different. 

The construction   works with necessary evident amendments in the set up of Theorems \ref{thloc}, \ref{thgen}, we comment on it in Remark \ref{rem:2}. The affine involutions $A_q$ will play important role in the proofs of  Theorems \ref{thloc}, \ref{thgen}.

In Section \ref{sec:2.2}  we recall, following \cite{MatTroy2012},  a geometric construction which associates   an Euclidean structure to a convex body.  The construction will
be useful  also in the proof of Theorem \ref{thloc}. 
The  affine involutions $A_q$ preserve the Euclidean structure and therefore are usual reflections, in the Euclidean structure,  with respect to hyperplanes.  Then,  $T$ is invariant with respect to the whole group  $O_n $  implying that it is an ellipsoid, see \ref{sec:2.2} for details.} 

\subsection{{ The existence of an affine  involution of $\mathbb{R}^{n*}$ preserving $T$}} \label{sec:2.1}

  Take 
  a point $q\in \partial K$ and  consider the mapping  $I:\partial T\to \partial T$ that 
	 sends  a point  $p_1\in \partial T \subset \mathbb{R}^{n*}$  to 
	the other point of the intersection of $\partial T$ with the line $\Lambda_q(p_1):=\{  p_1 + t \cdot d_qF \mid t\in \mathbb{R}\}$
 passing through $p_1$ and  generated\footnote{The differential $d_qF$ of the function $F$ at $q$ is a 1-form on $\mathbb{R}^n$,
 so it is a vector  of $\mathbb{R}^{n*}$}
	by   $d_qF\in \mathbb{R}^{n*}$. 
		At the points  $p_1\in \partial T$ in which $d_qF$  is tangent to $\partial T$, we set $ I(p_1)=p_1$.

		By construction, $I$ is a  well-defined involution. The set of its fixed  points is topologically 
		a $(n-2)$-sphere which we call {\it equator}. Equator divides $\partial T$ in two connected components, one of  which we call {\it the northern } and another  {\it the  southern hemispheres}\footnote{{ In the proof of Theorem \ref{thgen}, the role  of  the   hemispheres play $S_1$ and $S_2$}}. We  denote the northern hemisphere by  $S_+$. In our convention, the equator is not contained in the hemispheres; of course, it is contained in the closure of each hemisphere.

	Next, let us consider the endomorphism $H^*_q:\mathbb{R}^{n*} \to \mathbb{R}^{n*}$ 
 dual\footnote{That is, $H^*(p_1)(v)= p_1(H(v))$  for any 1-form $p_1\in \mathbb{R}^{n*}$ and any vector $v\in \mathbb{R}^n$} to $H_q$. 
 
 Observe  that $d_qF$ is an 
 eigenvector of the operator $H^*_q$  with eigenvalue $-1$. Indeed,  consider a basis in $\mathbb{R}^n$  such that its first element, we call it $\mcn$,  is an  eigenvector of $H_q$ with eigenvalue $-1$, and all other, we call them 
 $v_1,\dots, v_{n-1}$, are eigenvectors with eigenvalue $1$. By assumptions, the vectors  $v_1,\dots, v_{n-1}$ are tangent to $\partial K$ at $q$, so $d_qF(v_i)=0$. 
 Then, $$H^*_q(d_qF)(\mcn)= d_qF(H_q(\mcn))= -d_qF(\mcn) \ \textrm{and} $$ $$
 H^*_q(d_qF)(v_i)=d_qF(H_q(v_i))=d_qF(v_i)=0=- d_qF(v_i).$$ 
We see that  the 1-forms $-d_qF$ and $H^*_q(d_qF)$ have the same values on the basis vectors $\mcn, v_1, \dots, v_{n-1}$ and therefore coincide. 
Clearly, dual endomorphisms have the same eigenvalues and the same  dimensions of eigenspaces corresponding to each given eigenvalue, so $H^*_q$ has an eigenspace of dimension $n-1$  with eigenvalue $1$.

Consider now the open set $$\mathbb{T}=\bigcup_{p_1\in S_+}\Lambda_q(p_1) =  \{p_1 + t \cdot d_qF\mid  t \in \mathbb{R}, p_1\in S_+ \} \subset \mathbb{R}^{n*}, $$
which is the union of all the lines intersecting the northern hemisphere and 
parallel to the vector $d_qF$.  As $T$ is strictly convex, for each such a line
$\Lambda_q(p_1)$,  $p_1\in S_+$, its  intersection  with $S_+$  coincides with the point $p_1$.  On $\mathbb{T}$, let us define
the codimension one distribution   $D$ (i.e., subspace field) by  the  following conditions:
\begin{enumerate}
    \item The distribution is invariant with respect to the translation  
    $p_1\mapsto p_1 +  t \cdot d_qF$ with arbitrary $t$. 
    \item The northern hemisphere $S_+$ is tangent to the distribution 
    $D$, i.e., $d_{p_1}\wt F(D_{p_1})=0$ for every $p_1\in S_+$. See Figure 3. 
\end{enumerate}
The distribution is well defined on the set $\mathbb{T}$ and is at least continuous. It is integrable, since the northern hemisphere and all its translations with respect to $t \cdot d_qF$ are tangent to the distribution. In particular,  $D$ can be rectified by a  $C^1$-smooth diffeomorphism 
that fixes each line forming $\mathbb T$, acts along it as a translation and 
sends the northern hemisphere to a hyperplane. 

 Now, take the distribution $\wt D:= H_q^*D.$  That is, at every point of $\mathbb{T}$, the distribution $\wt D$ consists of all vectors of the form $H^*_q(\alpha)$ with $\alpha\in D$.
 
Observe that that the southern hemisphere is tangent to the distribution $\wt D$. This follows from Minkowski and projective billiard coincidence. Indeed, for every $p_1$ from the northern hemisphere  denote by  $ p_2$  the point of the southern hemisphere such that the line $p_1 p_2$ is parallel to $d_qF$. The projective billiard reflection coincides, up to reparametrization of orbits,  with the Minkowski billiard reflection. This implies that the vector $H_q(d_{p_1}\wt F)$ is proportional to $d_{ p_2}\wt F$. 
On the other hand, 
$$H_q(d_{p_1} \wt F)(\wt D)= d_{p_1}\wt F(H_q^*\wt D)=d_{p_1}\wt F(D)=0,$$
since $H_q^*$ is an involution. 
Therefore, $d_{ p_2}\wt F(\wt D)=0$, and the southern hemisphere is 
tangent to $\wt D$. See Figure 3.

\begin{figure}[ht]
  \begin{center}
   \epsfig{file=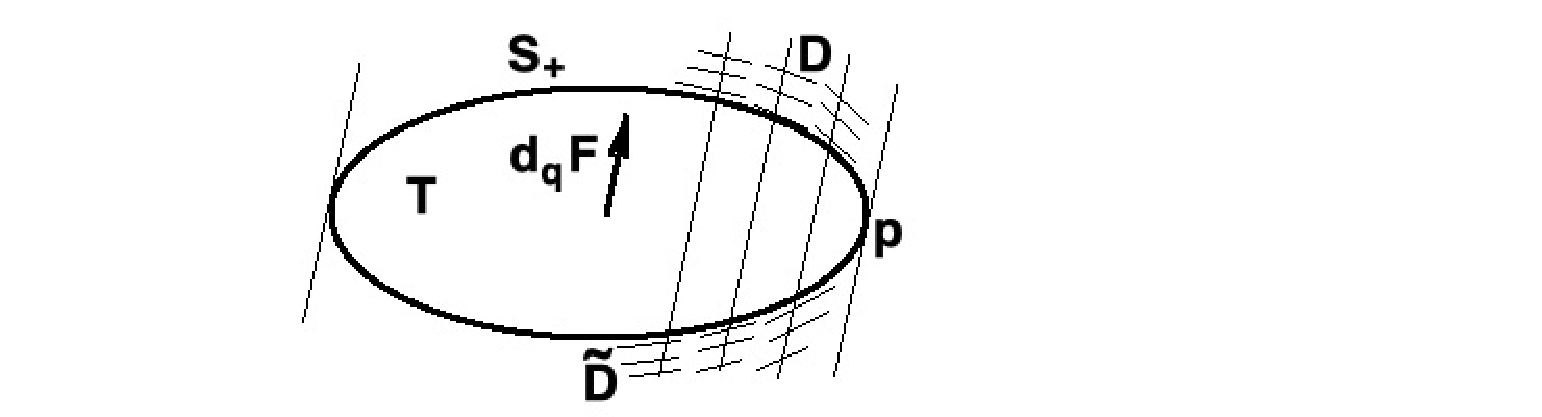, width=35em}
    \caption{The distributions $D$ and $\wt D$.}
    \label{fig:3}
  \end{center}
\end{figure}

Since the distribution $\wt D$ is also invariant with respect to  translations by vectors  $t \cdot d_qF$, it is integrable and its integral manifolds are translations of the southern hemisphere.

Next, take a point $p$ of the equator. 
Note that the only integral manifold of the distribution $D$  containing  the point $p$ 
is the northern hemisphere $S_+$. Similarly, the only integral manifold of the distribution $\wt D$  containing  the point $p $ is the southern hemisphere.

Now, consider the 
		affine transformation $$  A_q: \mathbb{R}^{n*} \to \mathbb{R}^{n*} \ ,  \ \   p_1 \mapsto H_q^*(p_1) - H_q^*(p) + p. $$ 
 This affine transformation  clearly sends $p$ to $ p$. 
Since $d_qF$ is  an eigenvector of $H^*_q$, the 
mapping $A_q$ induces a well defined 
  affine transformation on 
  the quotient of $\mathbb{R}^{n*}$ by the 
  translations with respect to vectors $t\cdot d_qF$.  The structure of eigenvalues and eigenspaces of $H_q^*$ and the property $A_q(p)=p$ imply that
  $A_q$ acts by identity on  the quotient space. Indeed, it sends the line  $\Lambda_q(p)$ to itself. Moreover, since the eigenspace corresponding to $-1$ is ``killed'' by taking the quotient,  the corresponding linear transformation has eigenvalue $1$ with eigenspace of dimension $n-1$, an  is  therefore the identity. 
Thus,  $A_q$ sends  every line $\Lambda_q(p_1)= \{p_1+ t d_qF \mid t\in \mathbb{R}\}$  to itself. In particular,  
$A_q$  is an affine involution.

As the linear mapping corresponding to $A_q$ is $H_q^*$ and $A_q$ sends  every line $\Lambda_q(p_1)$  to itself, 
   the pushforward 
 of $D$ with  respect to $A_q$ coincides with $\wt D.$   Then, $A_q$ sends the northern semisphere to the southern hemisphere, since  the northern hemisphere (resp. the southern hemisphere)  is 
 the only integral manifold of $D $  (resp. $\wt D$)
 whose closure  contains $p$, and since $A_q$ takes $p$ to itself and $D$ to $\wt D$. As $H^*_q$ is an involution, $A_q$ is also an involution.  Therefore,  $A_q$ permutes the southern and the southern hemispheres.

\begin{remark} \label{rem:2}
Up to now, we did not really use the condition that $T$ is compact. The proof above  works with necessary evident amendments in the  ``semi-local'' and   ``local'' cases. Namely,  under the assumptions and in the notations of Theorem \ref{thloc} (respectively, Theorem \ref{thgen}), we have shown that for 
every one-dimensional vector subspace $\alpha\subset \mathbb{R}^{n*}$ which  is sufficiently close to a subspace parallel to the tangent hyperplane  $T_{O_S}S$  (respectively, for any one-dimensional subspace 
which is sufficiently close to the subspace parallel to the line   $O_1O_2$),  the natural analog of the 
mapping $I_\alpha$   coincides with the  restriction of an affine involution to $S$ (respectively, $S_1\cup S_2)$.  This affine involution will be denoted by  $A_\alpha$, it is constructed by a subspace $\alpha$ and $S$ (respectively,     $S_1 \cup S_2$).  The subspace $\alpha$ and the point $q\in U$     are related as follows: at the point $q$, the differential of $F$ is parallel to $\alpha$. 

Note that the  linear mapping corresponding to this   affine mapping, which we call $H^*_\alpha$,   automatically  has  $n-1$-dimensional eigenspace with eigenvalue 
$1$   and the subspace $\alpha$ is an eigenline of $H^*_\alpha$
with eigenvalue $-1$. 

 We would like to emphasize  that the  definitions of $I_\alpha$, $A_\alpha$ 
and $H^*_\alpha$ above do not require $K$ (respectively, $U$).
\end{remark}
Theorems \ref{thm:main},  \ref{thloc} and \ref{thgen} are therefore reduced to the following problem(s):

\vspace{1ex}
\noindent{\it Suppose for the boundary  $\partial T$  of a strictly convex body (respectively, for a  local strictly convex hypersurface $S$, respectively  for $S_1 \cup S_2$)  and  for every one-dimensional vector subspace $\alpha\subset\rr^{n*}$ (respectively, for every $\alpha$ which is sufficiently close to a line parallel to the tangent hyperplane to $S$ at $O_S$, respectively,  for every $\alpha$ which is sufficiently close to $\alpha_0$ parallel to $O_1O_2$), the mapping $I_\alpha$ is the restriction of a certain affine mapping $A_\alpha$  to $\partial T$ (respectively, $S$, respectively $S_1 \cup S_2$). We need to prove that $\partial T$   is a quadric (respectively, $S$ is a part of a quadric, respectively  $S_1 \cup S_2$ is a part of a quadric, assuming that the union of any two hyperplanes is a quadric).} 

\vspace{1ex}

\subsection{{ $T$ is the  sphere in its Binet-Legendre metric} }  \label{sec:2.2}

From this point and until the end of this section we work under the assumptions of Theorem \ref{thm:main}, so  $T $ is a compact body. As 
$A_q$ permutes the northern and the southern hemispheres, it    preserves the convex body $T$.  
But then $A_q$ preserves the barycenter of $T$ which is the origin of $\mathbb{R}^{n*}$ by our assumptions. Then, $A_q$ is a linear map and therefore it coincides with $H_q^*$.

		 Next, consider the Binet-Legendre metric corresponding to $\wt F$.
   Recall that Binet-Legendre metric   is an
  Euclidean structure   canonically constructed in \cite{MatTroy2012} by  a convex body $T$  using the following   formula: for two vectors $v_1, v_2 \in \mathbb{R}^n$, their inner  product is given by 
\begin{equation} \label{eq:10}
    \langle v_1, v_2\rangle_{BL} = \frac{n+2}{\lambda(T)}\int_{T} \alpha(v_1) \alpha(v_2) d\lambda(\alpha),
\end{equation}  
where $\lambda$ is a Lebesgue measure on $\mathbb{R}^{n*}$, and $\lambda(V)$ is the volume of $T$ in this measure. In other words,     up to a constant, the Binet-Legendre metric is the 
$L^2$-scalar product of the vectors $v_1, v_2$ viewed as elements of $(\mathbb{R}^{n*})^*$ restricted to 
$T$. The   formula  \eqref{eq:10} is not affected by  the choice of the Lebesgue measure $\lambda$ (i.e., by its multiplication by constant)  and  defines  an inner product on $\mathbb{R}^{n}$. As an inner product on $\mathbb{R}^{n*}$ which we will use in our further considerations and also call Binet-Legendre metric, we simply take the  dual of the one given by \eqref{eq:10}. 

Note  that there exist many different geometric constructions of an Euclidean structure by a convex body, say the one coming from the so called John ellipsoid, see e.g. \cite{MTJohn}. See also \cite[\S 2.3.2]{MMP} for different  constructions using  Finsler Browning motion.  
The advantage of the formula \eqref{eq:10}  over certain other geometric constructions, e.g.,  over the  ``averaging'' construction from \cite{MRTZ, M09},   is that it requires no smoothness of $\partial T$ and behaves well under continuous  perturbations of the body $T$, see e.g.  \cite[Theorem B]{MT17}.

Since  $H_q^*$ preserves  
$T$,   it also preserves the Binet-Legendre metric.  Therefore, the mapping $H_q^*$ is simply the orthogonal 
reflection, in the Binet-Legendre metric, with respect to a  hyperplane containing the origin.

Recall  that up to now, in all arguments,  the point $q$ was a fixed  arbitrary chosen point of $\partial K$. The reflection $A= H_q^*$  was constructed by this point and has the property that  the hyperplane of the reflection is orthogonal, in the Binet-Legendre metric, to  $d_qF$.  
As we can choose the point $q$ and therefore $d_qF$ arbitrarily (strict convexity of the body $K$), $T$ is invariant with respect to the reflections  from all hyperplanes  containing the origin. Then, $T$   is the sphere
in the Euclidean structure given by the Binet-Legendre metric. Then, it is an ellipsoid in $\mathbb{R}^{n*}$ and Theorem \ref{thm:main} is proved. 

\section{Proof of Theorem \ref{thloc}} 
\subsection{Proof of Theorem \ref{thloc} for $n=2$} \label{sec:4}
In the case under consideration, $S$ is a strictly convex curve in $\rr^{2*}$. We will identify each tangent line to $S$ with the one-dimensional vector subspace $\alpha\subset\rr^{2*}$ parallel to it. 
For every tangent line $\alpha$ to $S$, the corresponding 
affine involution $A_\alpha$ from Remark \ref{rem:2} depends only 
on the line $\alpha$. 
Since $S$ is strictly convex, its points are in one-to-one correspondence with tangent lines and with the corresponding subspaces in $\rr^{2*}$. Therefore, from now on we identify a point in $S$ with the corresponding tangent line and consider  that $A_\alpha$ is a family of affine involutions parametrized by points $\alpha\in S$. Its dependence on $\alpha$ is at least  continuous.   By construction, $A_\alpha$ fixes $S$, and  it fixes $\alpha$, since it fixes the corresponding tangent line. 

Consider the Lie group $G$ of affine transformations of the plane 
$\rr^{2*}$, such that the correspoding linear transformations  have  determinants $\pm1$. The group $G$ clearly   contains all    $A_{\alpha}$. Its connected component $G^o$ of the identity consists of area-preserving affine transformations. Let $\gg$ denote its Lie algebra, which is the algebra of divergence-free linear non-homogeneous\footnote{That is,  its components are polynomials of degree $\le 1$ in standard coordinates of $\mathbb{R}^{2*}$} vector fields. The exponential map of the Lie group $G$  is a diffeomorphism of a neighborhood of the origin in its Lie algebra $\gg$ onto a neighborhood of the identity in $G$. 
\medskip

\begin{proposition}
    \label{claim} There exists a vector $v\in\gg$, $v\ne \vec 0$,  such that for every $t\in\rr$ small enough the element $X^t:=\exp(tv)\in G$ leaves the curve $S$ invariant. In other terms, $S$ is tangent to the divergence-free linear non-homogeneous vector field representing $v$.\end{proposition}

\begin{proof}
For every $\alpha\in S$ consider the composition 
$$B_\alpha:=A_\alpha\circ A_{O_S}: \ \ \ B_{\alpha}\to Id, \ \text{ as } \ \alpha\to O_S.$$
One has $B_{\alpha}\neq Id$ whenever $\alpha$ is close enough to 
$O_S$ and different from $O_S$: the transformation $A_{O_S}$ fixes $O_S$, while $A_{\alpha}$ doesn't. Thus, $B_{\alpha}\in G^o$, whenever 
$\alpha$ is close enough to $O_s$. On the way proving Propostion \ref{claim}, we need to  prove the following  proposition which we will use further on. 

\begin{proposition} \label{claim2} Let a family $B_\alpha:\rr^2\to\rr^2$ of planar area-preserving affine transformations depend continuously on a parameter $\alpha$ from a subset in $\rr^m$ accumulating to a point 
$\alpha_0$. Let $B_{\alpha_0}=Id$, and let $B_{\alpha}\neq Id$ for $\alpha\neq\alpha_0$. Let there exist a germ of $C^1$-smooth curve $S\subset\rr^2$ at a point $O$ such that for every $\alpha$ close enough to $\alpha_0$ 
the map $B_{\alpha}$ sends $S$ to itself. Then $S$ is a phase curve of a divergence-free linear non-homogeneous vector field. 
\end{proposition}
\begin{proof}
For every $\alpha$ close to $\alpha_0$ let $v_{\alpha}\in\gg\setminus\{0\}$ denote the preimage of the element $B_{\alpha}\in G$ under the exponential map. 
As $\alpha\to \alpha_0$, the family of one-dimensional subspaces  $\rr v_\alpha\subset\gg$  accumulates to at least one limit one-dimensional subspace: let us fix a limit subspace and denote it by 
$\ell$. Fix a vector $v$ generating $\ell$. It corresponds to a divergence free linear non-homogeneous vector field on $\rr^2$, which will be also denoted by $v$. Let us show that $S$ is its phase curve. 
Namely, we show that for every small $\var>0$ there exist $\delta,\sigma>0$ such that 
for every $t\in(-\delta,\delta)$ the time $t$ flow map $X^t$ of the field $v$ sends the intersection 
of the curve $S$ with the $\sigma$-neighborhood $U_\sigma(O)$ of the 
point $O$ to $U_{\var}(O)\cap S$. 

Indeed, fix a Euclidean quadratic form on $\gg$ in which $|v|=1$. 
Fix a small $\var>0$. Fix  $\delta,\sigma>0$ such that for every 
$w\in\gg$ with $|w|<\delta$ the image under the map $\exp(w)$ of the 
$\sigma$-neighborhood $U_{\sigma}(O)$ of the point $O$ in $\rr^{2*}$ lies in   $U_\var(O)$. 
Fix a sequence $\alpha_k\to\alpha_0$ such that the one-dimensional subspaces generated by $v_{\alpha_k}$ converge to $\ell$. For every 
$t\in(-\delta,\delta)$ fix 
a sequence of numbers $m_{k,t}\in\mathbb N$ such that 
$$w_{t,k}:=m_{k,t}v_{\alpha,k}\to tv, \ \text{ as } k\to\infty; \ \ \ 
|w_{t,k}|<|t|=|tv|.$$
It exists, since $\rr v_{\alpha,k}\to\rr v$, $|v_{\alpha,k}|\to0$ and $|v|=1$. 
Then one has 
$$Y_{t,k}:=B_{\alpha_k}^{m_{t,k}}=\exp(w_{t,k})\to X^t , \ \text{ as } k\to\infty.$$
The maps $B_{\alpha}$ preserve the curve $S$. For every $m=1,\dots,m(t,k)$ the power $B_{\alpha_k}^m$ sends $U_\sigma(O)$ to 
$U_{\var}(O)$. Therefore, it sends the intersection 
$U_{\sigma}(O)\cap S$ to $U_{\var}(O)\cap S$. Hence, so does the limit map $X^t$. This implies that the planar vector field $v$ is tangent to $S$ and finished the proof of Proposition \ref{claim2}.
\end{proof}

Proposition \ref{claim2} immediately implies the statement of Proposition \ref{claim}
\end{proof}

Thus, $S$ is a phase curve of a divergence-free linear non-homogeneous planar vector field $v$. 

\begin{proposition} \label{procon} Each phase curve of a divergence-free linear non-homogeneous planar vector field $v$ is either a (part of a) conic, or a (part of a) line.
\end{proposition}

We  believe that this proposition is well-known, but give a proof for self-containedness.  

\medskip

\begin{proof} The matrix of the linear terms of the field $v$ has trace $0$ and therefore 
  is conjugated to one of the following four  matrices:
$$\begin{pmatrix}\lambda &0 \\ 0& -\lambda  \end{pmatrix}\ , \ \ \begin{pmatrix}0 &-\lambda \\ \lambda & 0 \end{pmatrix}\ , \ \ \begin{pmatrix}0 &1 \\ 0 & 0 \end{pmatrix}\ , \ \  \begin{pmatrix} 0 &0 \\ 0& 0  \end{pmatrix}; \ \ \ \lambda\neq0.$$

 We treat each matrix  separately, and show that in each case the phase curve is as we claimed. We work in the  standard 
 coordinates $x,y$ on the plane.
\begin{itemize}
    \item[Case 1:]  The field $v$ can be transformed to the field $\tilde v=(x,-y)$   by an affine transformation. Phase curves of the latter field are half-hyperbolas and positive and negative  coordinate axes.

\item[Case 2:]  The field $v$  can be transformed to the field $\tilde v=(y,-x)$  by an affine transformation. The only singular point is the point $(0,0)$ and all nontrivial phase curves are circles.

\item[Case 3:]   By an affine transformation, the field $v$ can be transformed to the field $\tilde v =(y, c )$ with  $c \in \mathbb{R}$. 
If $c=0$, then the latter field has line $\{ y=0\}$ of singular points, and all its non-trivial phase curves are lines. If $c\neq0$, then each phase curve is a parabola.
\item[Case 4:]  The field $v$  has  constant entries and its phase curves are straight lines. 
\end{itemize}
\end{proof}

Two-dimensional version of 
Theorem \ref{thloc} immediately follows from Propositions \ref{claim}  and  \ref{procon}. Indeed, by Proposition \ref{claim}, 
the curve $S$ is  (a part of)  a phase curve of a divergence free linear non-homogeneous planar vector field $v$. Next,  by Proposition \ref{procon}, a strictly convex phase curve of a  divergence free linear non-homogeneous planar vector field $v$ is a conic.  

\subsection{ Proof of Theorem \ref{thloc} in dimension $n\ge 3$} \label{sec:5}

We consider the tangent hyperplane $T_{O_S}S$ to $S$ at   $O_S$ and view it as a hyperplane in $\mathbb{R}^{n*}$ passing through $O_S$ . 
Next, take a  hyperplane  $ P$  which is parallel to  $T_{O_S}S$, which is  sufficiently 
close to  $T_{O_S}S$,  and   such that the intersection of $P$ with $S$ is nonempty. Because of strict convexity, the intersection $P\cap S$ is then a boundary of a  compact 
convex body  which we denote by $T'$ and view as a convex body in $P\approx \mathbb{R}^{n-1}$.  We denote by $B$ the 
barycenter of this convex body  and think without loss of generality that $B$ is the origin in our coodinate system in $\mathbb{R}^{n*}$. 

Consider the Euclidean structure in $P$  given by the Binet-Legende metric constructed  by   $T'$. Next, for every one-dimensional vector subspace $\alpha\subset\rr^{n*}$ parallel to $T_{O_S}S$,   we consider the  corresponding affine  transformation $A_\alpha$, see Remark \ref{rem:2}. As  $A_\alpha$  sends $B$ to itself, it is a linear transformation which we called  $H^*_\alpha$
in Remark \ref{rem:2}.

The  transformation $H_\alpha^*$ 
preserves $T'$, so its restriction to $P$ preserves the Binet-Legendre metric in $P$ constructed by $T'$.  The restriction of $H_\alpha^*$   to $P$ coincides  with the reflection, in the Binet-Legendre metric, with respect to the $(n-2)$-dimensional hyperplane orthogonal in the Binet-Legendre metric to $\alpha$ and containing $B$. Since $\alpha \in T_{O_S}S$ can be chosen arbitrary,   such reflections generate   the standard action of the group of rotations $O(n-1)$,
in the Binet-Legendre metric,   
 around $B$.  As the transformations $A_\alpha=H^*_\alpha$, by construction, take the point $O_s$ to itself,   the transformations $H^*_\alpha$, viewed now as transformations of the whole $\mathbb{R}^{n*}$, generate the standard action of the group $O(n-1) \subset O(n)$ on $\mathbb{R}^{n*}$. The group $O(n-1)$ is  viewed as      the 
stabiliser of $O_S$ assuming that  $B$ is 
 the  origin of our  coordinates system in $\mathbb{R}^{n*}$. The    Euclidean structure in $\mathbb{R}^{n*}$  is defined as follows: its restriction to $P$ is the Binet-Legendre metric of $T'$  and the line connecting  $B$ and $O_S$ is orthogonal to $P$.  

As all transformations $H^*_\alpha$ preserve  $S$, we have that $S$  is   rotationally  symmetric with respect to this action of $O(n-1)$.

We consider a two-dimensional plane containing $O_S$ and
$B$. Let $\Gamma$ denote its intersection with $S$. For every one-dimensional vector subspace 
$\alpha$ parallel to a line tangent to $\Gamma$ the corresponding involution $A_{\alpha}$ 
preserves $\Gamma$, since $\alpha$ lies in the above plane: see Section 2.  This together with the  proof of the ($n=2$)-dimensional case of Theorem \ref{thloc} given in   
the previous section  
yield that $\Gamma$ is a conic. It is clearly symmetric with respect to the orthogonal reflection from the line connecting  $B$ and $O_S$.  Then, the rotational hypersurface obtained by rotations of  the conic $\Gamma$ by all  elements of  the group $O(n-1)$ is a quadric and Theorem \ref{thloc} is proved. 

\section{ Proof of Theorem  \ref{thgen}}

The distribution arguments from the proof of Theorem \ref{thm:main}, see Section 2 and Remark \ref{rem:2}, 
imply that,  under  the assumptions  of Theorem \ref{thgen},  for every $\alpha$ close enough to $\alpha_0$,  the involution  $I_\alpha$ is the restriction to 
$S_1\cup S_2$ of a global affine involution $A_\alpha:\rr^{n*}\to\rr^{n*}$. This reduces Theorem \ref{thgen} to the following Proposition \ref{thgn}, which we prove in Sections \ref{sec:51} and \ref{sec:6}.

\begin{proposition}\label{thgn} Let  $S_1,S_2\subset \rr^{n*}$ be two distinct germs of hypersurfaces at  points $O_1$ and $O_2$ respectively. Let $\alpha_0\subset\rr^{n*}$ be a one-dimensional vector subspace. Let the line through $O_1$,   parallel to  $\alpha_0$,  
pass through $O_2$ and be transversal to $S_1$ and $S_2$. 
Let for every 
$\alpha$ close enough to $\alpha_0$ the involution $I_\alpha$ be the restriction to $S_1\cup S_2$ of an affine involution. Then $S_1$ and $S_2$ are parts of the same quadric (allowed to be a union of distinct hyperplanes).
\end{proposition}
The proof is organised as follows:   In Section \ref{sec:51}  we show that 
$S_1$, $S_2$ are parts of either hyperplanes, or quadrics. Then,  in 
Section \ref{sec:6},  we show that if $S_1, S_2$ are parts of quadrics,  they are parts of the same quadric.

\subsection{Proof that $S_1$ and  $S_2$ are parts of quadrics} \label{sec:51} 

Let us first consider the two-dimensional case: $S_1$ and $S_2$ are germs of planar curves transversal to the line through $O_1$ and $O_2$ 
parallel to $\alpha_0$. For every $\alpha$ close enough to $\alpha_0$ 
the composition $B_{\alpha}=A_{\alpha}\circ A_{\alpha_0}$ is area-preserving and leaves 
each curve $S_j$ invariant and tends to the identity, as $\alpha\to\alpha_0$. This together with  Proposition \ref{claim2} 
in Section 3.1 implies that $S_1$, $S_2$ are parts of phase curves of  divergence-free linear non-homogeneous vector fields. 
Each  phase curve is either (a part of) a line, or (a part of) a  conic, by Proposition \ref{procon}. This proves that $S_1$, $S_2$ are parts 
of  conics (some of them may be degenerate, i.e., a union of two lines). 

Let us prove the similar statement in higher dimensions. The same argument, as in Section 2, implies that  for every 
one-dimensional subspace $\alpha$ close enough to $\alpha_0$ the affine involution 
$A_\alpha$ fixes each line parallel to $\alpha$. Hence, it fixes each plane 
$\Pi$ parallel to $\alpha$. Thus, the intersections $\Pi\cap S_1$, $\Pi\cap S_2$ are parts of conics, by the above argument applied to these intersections. This together with the Proposition \ref{thcquad} below   implies that $S_1$, $S_2$ are parts of quadrics.

\begin{proposition} \label{thcquad} 
Let  $S\subset\rr^n$ be a germ of $C^1$-smooth hypersurface at a point $O$. 
Let $\Pi_0\subset\rr^n$ be a given plane through $O$ transversal to $S$. 
Let for every plane $\Pi$ close enough to $\Pi_0$ the intersection $S\cap\Pi$ be a part of a conic. Then $S$ is a part of either a hyperplane, or a quadric. 
\end{proposition}

\begin{remark} \label{remanal} If  $S$ is $C^2$-smooth, Proposition \ref{thcquad} follows from 
a stronger result of M.\,Berger \cite[section 3]{berger}, see also  \cite[theorem 4]{art} 
and \cite[theorem 3.7]{Glu24}, where a proof of Propostion \ref{thcquad} was given in the $C^2$-smooth case. (Below we represent this proof for self-containedness.)  We prove that in fact, in the conditions of Proposition \ref{thcquad} the hypersurface $S$ is $C^\infty$-smooth and even analytic. This together with the above-mentioned already proved $C^2$-smooth case will prove Proposition \ref{thcquad} in full generality. 
\end{remark} 

\begin{proof} {\bf of Proposition  \ref{thcquad}} Without loss of generality we consider that 
$S$ is not a part of a hyperplane. We prove Proposition \ref{thcquad} by induction in $n$. 

Induction base: $n=2$. Then $S$ is obviously is a part of a conic, by assumption. 

Induction step. Let we have already proved Proposition \ref{thcquad} for $n\leq k$. Let us 
prove it for $n=k+1$. 
By the above remark, it suffices to show that  $S$ is analytic. To do this, fix a hyperplane $V_0$ through $O$ containing $\Pi_0$ and a plane $W_0$ through $O$ that is transversal to $V_0$ and close to 
 the plane $\Pi_0$ so that $W_0$ is transversal to $S$. Fix 
five distinct hyperplanes $V_1,\dots,V_5$  parallel and close enough to $V_0$ that do not pass through $O$. The intersections $\Gamma_j:=V_j\cap S$ 
are parts of quadrics (or hyperplanes) in $V_j$, by the induction hypothesis. 

Every plane $W$ parallel and close to $W_0$ is transversal to $S$, and the intersection 
$\gamma_W:=W\cap S$ is a regular curve that is a part of either a line, or a conic, by assumption. The curve 
$\gamma_W$ is transversal to each quadric $\Gamma_j$: $W$ is transversal to $V_j$, since 
$W_0$ is transversal to $V_0$. For every $j=1,\dots,5$ set 
$$P_{j,W}:=\gamma_W\cap\Gamma_j=W\cap\Gamma_j.$$
For every $j$ the plane $W$ is uniquely determined by $P_{j,W}$, and the inverse function 
$W(P_j)$ is given by the plane through $P_j$ parallel to $W_0$: $W(P_j)$ 
 depends analytically on the point $P_j=P_{j,W}\in\Gamma_j$. Therefore, $P_{j,W}$ 
depends analytically on $P_1=P_{1,W}$ as an implicit function. Now we consider  $P_j$, $j=2,\dots,5$ 
as analytic $\Gamma_j$-valued functions of $P_1$. For every $P_1\in\Gamma_1$  the curve 
$\gamma_{W(P_1)}$ also depends analytically on $P_1$, being the unique conic in the plane 
$W(P_1)$ passing through the given five points $P_1,\dots,P_5$ depending analytically on 
$P_1$. The curves $\gamma_{W(P_1)}$ corresponding to distinct $P_1$ are disjoint, since 
distinct  planes $W$ parallel to $W_0$ are disjoint. They cover a neighborhood of the base point $O$. This yields a local analytic chart on the latter neighborhood. In more detail, 
fix some linear coordinates $(x_1,x_2)$ on the plane $W_0$ centered at $O$ 
so that the $x_1$-axis be tangent to $S$ at $O$.  Let us complement $(x_1,x_2)$ to 
global affine coordinate system $(x_1,\dots,x_n)$ on $\rr^n$. 
The projection of each curve $\gamma_W$ to the $x_1$-axis induces its one-to-one parametrization by the $x_1$-coordinate. For every point $y\in S$ close to $O$ we set 
$W=W(y)$ to be the plane through $y$ parallel to $W_0$. To each $y$ we associate the pair 
$(P_1, x_1)\in\Gamma_1\times\rr$, where $P_1=P_1(y)=W\cap\Gamma_1$ and $x_1=x_1(y)$. Thus, 
 $(P_1,x_1)$ are analytic coordinates on a neighborhood of the base point $O$ on $S$, by 
 the above argument. Hence, the germ $S$ is analytic and hence, $C^2$-smooth. This together with 
 Remark \ref{remanal} implies that $S$ is a quadric. 
 
 Let us present the proof of the above implication, i.e., proof of Proposition \ref{thcquad} in the $C^2$-smooth case given in \cite[theorem 4]{art}, \cite[theorem 3.7]{Glu24}. Let $\ell$ denote the intersection of the plane $\Pi_0$ with the tangent hyperplane to $S$ at $O$. We can choose two distinct points $x,y\in S$ arbitrarily close to $O$ with the line 
 $xy$ being transversal to $S$ at $x$ and $y$ and arbitrarily close to  $\ell$, $O\notin xy$,  so that the plane 
 $\Pi_1$ through $x$, $y$, $O$ be arbitrarily close to $\Pi_0$. 
 This is possible, since $S$ is not a part of hyperplane. Then $\Pi_1$  intersects $S$ by a non-linear conic, by assumption. Below we show that there exists a quadric $C$ 
 through $x$ and $y$ that is tangent to $S$ there and has the same 2-jet at $x$, as $S$. 
 
 We claim that $S=C$. Indeed, each plane $\Pi$ through the points $x$ and $y$ that is close to $\Pi_1$ intersects $S$ and $C$ by conics 
 tangent to each other at $x$ and $y$ and having contact of order greater than two (i.e., at least 3) at $x$. Thus, the intersection index of the latter conics in $\Pi$ is no less than $3+2=5$. Hence, the conics coincide, by B\'ezout Theorem. Thus, $\Pi\cap S=\Pi\cap C$. 
 Since  $\Pi$ was arbitrary plane through $x$, $y$ close to $\Pi_1$, this implies that $O\in C$ and the germs of the hypersurfaces $S$ and $C$ 
 at $O$ coincide. Thus, $S$ is a quadric. 

 Now it remains to construct the above quadric $C$ ``osculating" with $S$ at $x$. Applying a projective transformation acting on  $\rp^n\supset\rr^n$ 
 we can and will consider that the hyperplane tangent to $S$ at $y$ is the infinity hyperplane. Let us choose affine  coordinates $(z_1,\dots,x_n)$ on its complementary affine space $\rr^n$ centered at $x$ so that the tangent hyperplane 
 to $S$ at $x$ be the coordinate $(z_1,\dots,z_{n-1})$-hyperplane, and 
 the line $xy$ be the $z_n$-axis. Then the second jet of the hypersurface $S$ at the origin $x$ is represented by the quadric  
 $C=\{ z_n=Q(z_1,\dots,z_{n-1})\}$ where $Q$ is a quadratic form. 
 The latter quadric $C$ is the one we are looking for, since it is tangent to the infinity hyperplane at $y$. 
The induction step, and therefore Proposition  \ref{thcquad} are  proved.
 \end{proof} 
 
In the setup of Proposition   \ref{thgn}, Proposition  \ref{thcquad} implies that $S_1$ and $S_2$ are parts of quadrics. In the next Section \ref{sec:6}, we show that they are parts of the same quadric. 

\subsection{Proof that $S_1$ and $S_2$ lie in the same quadric} \label{sec:6}

The remaining part of the  proof of Proposition \ref{thgn} is  to show that 
$S_1$ and $S_2$ are parts of the same quadric. Without loss of generality we consider that 
$S_1$ and $S_2$ are not both parts of hyperplanes, since a union of hyperplanes is a quadric. Then none of $S_1$, $S_2$ is a part of hyperplane, since 
the affine transformation $A_{\alpha_0}$, which permutes $S_1$ and $S_2$, cannot permute a hyperplane and a non-linear quadric. 

Suppose the contrary: 
$S_1$ and $S_2$  lie in two distinct non-linear quadrics, also denoted by $S_1$ and $S_2$.   There exists an open subset  $V\subset\rp^{n-1}$ such that for every $\alpha\in V$ there exists an affine involution $A_\alpha:\rr^{n*}\to\rr^{n*}$ sending an open subset of 
the quadric $S_1$ to an open subset of the quadric $S_2$. This implies that $A_\alpha$ permutes the quadrics $S_1$ and $S_2$ for an open subset of values $\alpha$. By analyticity, the property of extension of the map $I_\alpha$ up to a global affine transformation extends analytically, as $\alpha$ varies in complex domain. For  a generic point $p_1\in S_1$ (and in its complexification) the hyperplane $H$ tangent to $S_1$ at $p_1$ is not tangent to the complex quadric $S_2$. Indeed, otherwise the duals to the quadrics $S_1$, $S_2$ would coincide, and hence, $S_1=S_2$, -- a contradiction. Given a point $p_1$ as above, $p_1\notin S_2$, a generic line $L$ tangent to $S_1$ at $p_1$ (i.e., a generic line $L$ through $p_1$ lying in $H$) is not tangent to $S_2$. Or equivalently, it is not tangent to 
the hyperplane section $S_2\cap H$: for every point $p_1$ lying outside a given quadric in $H$ (e.g., our hyperplane section) a generic line through $p_1$ is not tangent to it.  We can choose the above $L$ being parallel to a one-dimensional subspace $\alpha$ arbitrarily close to $\alpha_0$. 
Then the involution $I_\alpha$ sends two points $p_2,p_2'$ of the intersection $L\cap S_2$ 
to $p_1$. Therefore, it cannot be restriction to $S_1\cap S_2$ of an affine involution. The contradiction thus obtained proves that 
$S_1=S_2$ and finishes the proof of Proposition \ref{thgn} and, therefore,  of Theorem   \ref{thgen}.

\printbibliography

@article {art,
    AUTHOR = {Arnold, Maxim and Tabachnikov, Serge},
     TITLE = {Remarks on {J}oachimsthal integral and {P}oritsky property},
   JOURNAL = {Arnold Math. J.},
  FJOURNAL = {Arnold Mathematical Journal},
    VOLUME = {7},
      YEAR = {2021},
    NUMBER = {3},
     PAGES = {483--491},
      ISSN = {2199-6792,2199-6806},
   MRCLASS = {37C83 (51M10)},
  MRNUMBER = {4293876},
MRREVIEWER = {Serge\ E.\ Troubetzkoy},
       DOI = {10.1007/s40598-021-00180-0},
       URL = {https://doi.org/10.1007/s40598-021-00180-0},
}

@article {MRTZ,
    AUTHOR = {Matveev, Vladimir S. and Rademacher, Hans-Bert and Troyanov, Marc and
              Zeghib, Abdelghani},
     TITLE = {Finsler conformal {L}ichnerowicz-{O}bata conjecture},
   JOURNAL = {Ann. Inst. Fourier (Grenoble)},
  FJOURNAL = {Universit\'e{} de Grenoble. Annales de l'Institut Fourier},
    VOLUME = {59},
      YEAR = {2009},
    NUMBER = {3},
     PAGES = {937--949},
      ISSN = {0373-0956,1777-5310},
   MRCLASS = {53C60},
  MRNUMBER = {2543657},
MRREVIEWER = {Andrew\ Bucki},
       DOI = {10.5802/aif.2452},
       URL = {https://doi.org/10.5802/aif.2452},
}

@article {M09,
    AUTHOR = {Matveev, Vladimir S.},
     TITLE = {Riemannian metrics having common geodesics with {B}erwald
              metrics},
   JOURNAL = {Publ. Math. Debrecen},
  FJOURNAL = {Publicationes Mathematicae Debrecen},
    VOLUME = {74},
      YEAR = {2009},
    NUMBER = {3-4},
     PAGES = {405--416},
      ISSN = {0033-3883,2064-2849},
   MRCLASS = {53C60 (53C22)},
  MRNUMBER = {2521384},
MRREVIEWER = {Zden\v ek\ Du\v sek},
       DOI = {10.5486/pmd.2009.4458},
       URL = {https://doi.org/10.5486/pmd.2009.4458},
}

@article {MT17,
    AUTHOR = {Matveev, Vladimir S. and Troyanov, Marc},
     TITLE = {The {M}yers-{S}teenrod theorem for {F}insler manifolds of low
              regularity},
   JOURNAL = {Proc. Amer. Math. Soc.},
  FJOURNAL = {Proceedings of the American Mathematical Society},
    VOLUME = {145},
      YEAR = {2017},
    NUMBER = {6},
     PAGES = {2699--2712},
      ISSN = {0002-9939,1088-6826},
   MRCLASS = {53B40 (35B65 53C60)},
  MRNUMBER = {3626522},
MRREVIEWER = {Behroz\ Bidabad},
       DOI = {10.1090/proc/13407},
       URL = {https://doi.org/10.1090/proc/13407},
}

@article {berger,
    AUTHOR = {Berger, Marcel},
     TITLE = {Seules les quadriques admettent des caustiques},
   JOURNAL = {Bull. Soc. Math. France},
  FJOURNAL = {Bulletin de la Soci\'et\'e{} Math\'ematique de France},
    VOLUME = {123},
      YEAR = {1995},
    NUMBER = {1},
     PAGES = {107--116},
      ISSN = {0037-9484,2102-622X},
   MRCLASS = {53A07 (53B50 58C27 78A05)},
  MRNUMBER = {1330789},
MRREVIEWER = {D.\ R. J. Chillingworth},
       URL = {http://www.numdam.org/item?id=BSMF_1995__123_1_107_0},
}

@article {MatTroy2012,
    AUTHOR = {Matveev, Vladimir S. and Troyanov, Marc},
     TITLE = {The {B}inet-{L}egendre metric in {Finsler geometry}},
   JOURNAL = {Geom. Topol.},
  FJOURNAL = {Geometry \& Topology},
    VOLUME = {16},
      YEAR = {2012},
    NUMBER = {4},
     PAGES = {2135--2170},
      ISSN = {1465-3060,1364-0380},
   MRCLASS = {53C60 (30C20 53A30 53C25 58B20)},
  MRNUMBER = {3033515},
MRREVIEWER = {Andrew Bucki},
      DOI = {10.2140/gt.2012.16.2135},
       URL = {https://doi.org/10.2140/gt.2012.16.2135},
}

@article {Tab97,
    AUTHOR = {Tabachnikov, Serge},
     TITLE = {Introducing projective billiards},
   JOURNAL = {Ergodic Theory Dynam. Systems},
  FJOURNAL = {Ergodic Theory and Dynamical Systems},
    VOLUME = {17},
      YEAR = {1997},
    NUMBER = {4},
     PAGES = {957--976},
      ISSN = {0143-3857,1469-4417},
   MRCLASS = {58F05 (58F11)},
  MRNUMBER = {1468110},
MRREVIEWER = {Nikolai\ Chernov},
       DOI = {10.1017/S0143385797086239},
       URL = {https://doi.org/10.1017/S0143385797086239},
}

@article {GT02,
    AUTHOR = {Gutkin, Eugene and Tabachnikov, Serge},
     TITLE = {Billiards in {F}insler and {M}inkowski geometries},
   JOURNAL = {J. Geom. Phys.},
  FJOURNAL = {Journal of Geometry and Physics},
    VOLUME = {40},
      YEAR = {2002},
    NUMBER = {3-4},
     PAGES = {277--301},
      ISSN = {0393-0440,1879-1662},
   MRCLASS = {37J99 (53B40 70F99)},
  MRNUMBER = {1866992},
MRREVIEWER = {Roberto\ Markarian},
       DOI = {10.1016/S0393-0440(01)00039-0},
       URL = {https://doi.org/10.1016/S0393-0440(01)00039-0},
}

@article{Glu24,
      title={On Hamiltonian projective billiards on boundaries of products of convex bodies}, 
      author={Alexey Glutsyuk},
      year={2024},
      journal = {To appear in Proc. Steklov Inst. Math.},
      eprint={2405.13258},
      archivePrefix={arXiv}
}

@article {AO2,
    AUTHOR = {Artstein-Avidan, Shiri and Ostrover, Yaron},
     TITLE = {Bounds for {M}inkowski billiard trajectories in convex bodies},
   JOURNAL = {Int. Math. Res. Not. IMRN},
  FJOURNAL = {International Mathematics Research Notices. IMRN},
      YEAR = {2014},
    NUMBER = {1},
     PAGES = {165--193},
      ISSN = {1073-7928,1687-0247},
   MRCLASS = {52A99 (37D50 37J10 70G45 82C05)},
  MRNUMBER = {3158530},
MRREVIEWER = {Daniel\ Massart},
       DOI = {10.1093/imrn/rns216},
       URL = {https://doi.org/10.1093/imrn/rns216},
}

@article {viterbo,
    AUTHOR = {Viterbo, Claude},
     TITLE = {Metric and isoperimetric problems in symplectic geometry},
   JOURNAL = {J. Amer. Math. Soc.},
  FJOURNAL = {Journal of the American Mathematical Society},
    VOLUME = {13},
      YEAR = {2000},
    NUMBER = {2},
     PAGES = {411--431},
      ISSN = {0894-0347,1088-6834},
   MRCLASS = {53D12 (49J45 53D05 53D35 57R17)},
  MRNUMBER = {1750956},
MRREVIEWER = {Darko\ Milinkovi\'c},
       DOI = {10.1090/S0894-0347-00-00328-3},
       URL = {https://doi.org/10.1090/S0894-0347-00-00328-3},
}

@article {ako1,
    AUTHOR = {Artstein-Avidan, Shiri and Karasev, Roman and Ostrover, Yaron},
     TITLE = {From symplectic measurements to the {M}ahler conjecture},
   JOURNAL = {Duke Math. J.},
  FJOURNAL = {Duke Mathematical Journal},
    VOLUME = {163},
      YEAR = {2014},
    NUMBER = {11},
     PAGES = {2003--2022},
      ISSN = {0012-7094,1547-7398},
   MRCLASS = {52A40 (37J05 52A23 53D35)},
  MRNUMBER = {3263026},
MRREVIEWER = {Dmitry\ Ryabogin},
       DOI = {10.1215/00127094-2794999},
       URL = {https://doi.org/10.1215/00127094-2794999},
}

@misc{ostrover,
    author = {Pazit Haim-Kislev 
              and Yaron Ostrover},
     title = {A Counterexample to Viterbo’s Conjecture},
     year = {2024},
   eprint={2405.16513},
      archivePrefix={arXiv},
}

@article {MMP,
    AUTHOR = {Ma, Tianyu and Matveev, Vladimir S. and Pavlyukevich, Ilya},
     TITLE = {Geodesic random walks, diffusion processes and brownian motion
              on finsler manifolds},
   JOURNAL = {J. Geom. Anal.},
  FJOURNAL = {Journal of Geometric Analysis},
    VOLUME = {31},
      YEAR = {2021},
    NUMBER = {12},
     PAGES = {12446--12484},
      ISSN = {1050-6926,1559-002X},
   MRCLASS = {53B40 (53C60 82B41 82C41)},
  MRNUMBER = {4322573},
MRREVIEWER = {Camelia\ M.\ Frigioiu},
       DOI = {10.1007/s12220-021-00723-z},
       URL = {https://doi.org/10.1007/s12220-021-00723-z},
}

@article {MTJohn,
    AUTHOR = {Matveev, Vladimir S. and Troyanov, Marc},
     TITLE = {Completeness and incompleteness of the {B}inet-{L}egendre
              metric},
   JOURNAL = {Eur. J. Math.},
  FJOURNAL = {European Journal of Mathematics},
    VOLUME = {1},
      YEAR = {2015},
    NUMBER = {3},
     PAGES = {483--502},
      ISSN = {2199-675X,2199-6768},
   MRCLASS = {53B40 (53C60)},
  MRNUMBER = {3401902},
MRREVIEWER = {Akbar\ Tayebi},
       DOI = {10.1007/s40879-015-0046-4},
       URL = {https://doi.org/10.1007/s40879-015-0046-4},
}

\bigskip

{A.\,G.:
CNRS, UMR 5669 (UMPA, ENS de Lyon), France,

HSE University, Moscow, Russia,

Higher School of Modern Mathematics MIPT, Moscow, Russia,

\url{aglutsyu@ens-lyon.fr}}

\medskip

{V.\,M.:
Institut f\"ur Mathematik, 
Friedrich Schiller Universit\"at Jena,

07737 Jena, Germany, \, \url{ vladimir.matveev@uni-jena.de}}

\end{document}